\documentclass[12pt]{amsart}
\usepackage[english]{babel}
\usepackage[utf8x]{inputenc}
\usepackage{amsfonts}
\usepackage{amsmath}
\usepackage{longtable}
\usepackage{graphicx}
\usepackage{enumerate}
\usepackage{tikz-cd}
\usepackage{bbm}
\usepackage{mathtools}
\usepackage{amsthm}
\usepackage{amssymb}
\usepackage{url,smartref}
\usepackage{mathrsfs} 
\usepackage{tikz}
\usepackage{extpfeil,bm,fancyhdr,mathbbol,comment}
\usepackage[top=1in, bottom=1.25in, left=1.25in, right=1.25in]{geometry}
\usetikzlibrary{shapes, arrows}
\tikzstyle{terminator} = [rectangle, draw, text centered, rounded corners, minimum height=2em]

\usepackage{enumitem, hyperref}
\makeatletter
\def\namedlabel#1#2{\begingroup
    #2%
    \def\@currentlabel{#2}%
    \phantomsection\label{#1}\endgroup
}
\makeatother

\makeatletter
\newcommand{\labitem}[2]{%
\def\@itemlabel{\textbf{#1}}
\item
\def\@currentlabel{#1}\label{#2}}
\makeatother

\makeatletter
\def\namedlabel#1#2{\begingroup
   \def\@currentlabel{#2}%
   \label{#1}\endgroup
}
\makeatother

\subjclass[2020]{11R23 (primary); 11G05, 11R20 (secondary) }
\keywords{Iwasawa theory, Hecke characters, Hilbert modular forms, Mordell--Weil groups, ordinary primes, anticyclotomic extensions}

\numberwithin{equation}{section}

\usepackage{multirow}
\usepackage{subfigure}
\usepackage{algorithm}
\usepackage{algorithmic}
\usepackage{comment,mathrsfs}

\title{On anticyclotomic Iwasawa Theory of Hecke Characters at ordinary primes}
\author{Erman I{\c S}IK}

\address{University of Ottawa, Department of Mathematics and Statistics, STEM Complex, 150 Louis-Pasteur Pvt, Ottawa, ON, K1N 6N5, Canada}
\email{\url{eisik@uottawa.ca}}
\urladdr{\url{https://sites.google.com/view/erman-isik/}}

\newcommand{\mylabel}[2]{#2\def\@currentlabel{#2}\label{#1}}
\DeclareFontFamily{U}{wncy}{}
    \DeclareFontShape{U}{wncy}{m}{n}{<->wncyr10}{}
    \DeclareSymbolFont{mcy}{U}{wncy}{m}{n}
    \DeclareMathSymbol{\Sh}{\mathord}{mcy}{"58}

\theoremstyle{definition}
\newtheorem{definition}{Definition}[section]
\newtheorem{lemma}[definition]{Lemma}
\newtheorem{theorem}[definition]{Theorem}
\newtheorem{prop}[definition]{Proposition}

\newtheorem{conj}[definition]{Conjecture}

\newtheorem{remark}[definition]{Remark}
\newtheorem{example}[definition]{Example}
\newtheorem{hyp}[definition]{Hypothesis}
\newtheorem{thmx}{Theorem}

\DeclareFontFamily{U}{wncy}{}
    \DeclareFontShape{U}{wncy}{m}{n}{<->wncyr10}{}
    \DeclareSymbolFont{mcy}{U}{wncy}{m}{n}
    \DeclareMathSymbol{\Sh}{\mathord}{mcy}{"58}
\usepackage[OT2,T1]{fontenc}
\DeclareSymbolFont{cyrletters}{OT2}{wncyr}{m}{n}
\DeclareMathSymbol{\Sha}{\mathalpha}{cyrletters}{"58}

\newcommand{\ccF}{\mathcal{F}}

\newcommand{\ccO}{\mathcal{O}}
\newcommand{\ccL}{\mathcal{L}}
\newcommand{\ccW}{\mathcal{W}}

\newcommand{\bbA}{{\mathbb A}}

\newcommand{\bbC}{{\mathbb C}}

\newcommand{\bbQ}{{\mathbb Q}}

\newcommand{\bbZ}{{\mathbb Z}}
\newcommand{\bbN}{{\mathbb N}}
\newcommand{\bbT}{{\mathbb T}}
\newcommand{\bbX}{{\mathbb X}}

\newcommand{\bfR}{{\mathbf R}}

\newcommand{\bbAt}{{\mathbb A}^{\times}}

\newcommand{\fraa}{{\mathfrak a}}

\newcommand{\fraf}{{\mathfrak f}}

\newcommand{\frah}{{\mathfrak h}}

\newcommand{\fran}{{\mathfrak n}}

\newcommand{\frap}{{\mathfrak p}}
\newcommand{\fraq}{{\mathfrak q}}

\newcommand{\frakF}{{\mathfrak F}}

\newcommand{\frakO}{{\mathfrak O}}
\newcommand{\frakP}{{\mathfrak P}}

\newcommand{\gal}{{\rm Gal}}

\newcommand{\Hom}{{\rm Hom}}

\newcommand{\Sel}{{\rm Sel}}

\newcommand{\ac}{{\rm ac}}
\newcommand{\cyc}{{\rm cyc}}
\newcommand{\ind}{{\rm Ind}}
\newcommand{\im}{{\rm Im}}

\newcommand{\can}{{\rm can}}

\newcommand{\Char}{{\rm char}}

\begin{document}
\maketitle

\begin{abstract}
   In this article we study the Iwasawa theory for Hecke characters associated with CM abelian varieties and Hilbert modular forms at ordinary primes.  We formulate and prove a result concerning the anticyclotomic Iwasawa main conjecture for CM Hilbert modular forms. Additionally, we obtain a result towards the study of the Mordell--Weil ranks of the CM abelian varieties.
\end{abstract}

\section{Introduction}

The main objective of this paper is to study the $p$-ordinary Iwasawa theory for Hecke characters attached to CM abelian varieties and Hilbert modular forms and thereby to extend the results of Agboola--Howard \cite{AH06}, Arnold \cite{AR07} and B{\"u}y{\"u}kboduk \cite{buy14}.

\vspace{1.2mm}

Let $K$ be a CM field with maximal real subfield $F$ of degree $2g$, and let $p$ be an odd rational prime unramified in $F/\bbQ$. Assume that the $p$-ordinary condition of Katz holds:
\begin{equation}
 \tag{$p$\textbf{-ord}} \label{eqn:ordinaryassumption}  \text{Every prime of } F \text{ above } p \text{ splits in } K. 
\end{equation}

When the Hecke character in question is attached to an elliptic curve $E/\bbQ$ with complex multiplication by an imaginary quadratic field $K$ and the sign of the functional equation is $+1$, the $p$-ordinary Iwasawa main conjecture for $E$ along the anticyclotomic $\bbZ_p$-extension of $K$ follows from the two-variable main conjecture proved by Rubin in \cite{Rub91}. When the sign is $-1$, Agboola and Howard proved in \cite{AH06} the anticyclotomic main conjecture and Arnold generalized their result in \cite{AR07} to higher-weight CM forms. The key ingredients in all these cases are the elliptic unit Euler system and the reciprocity laws of Coates--Wiles \cite{CW77,wil78} and Kato \cite{kato99}, as well as Rubin's two-variable main conjecture.  

\vspace{1.2mm}

 Rubin's result on the Iwasawa main conjecture for imaginary quadratic fields has been extended to various settings by many people utilizing different techniques \cite{ht94, hid09, mai08, hsi14, buy14}. For instance, in \cite{ht94}, Hida and Tilouine used the CM ideal method, exploiting congruences between CM forms and non-CM forms to prove the anticyclotomic Iwasawa main conjecture for general CM fields. The work of Hsieh in \cite{hsi14} on the main conjecture along the maximal $\bbZ_p$--power extension relies on the congruences between the Eisenstein series and cusp forms. B{\"u}y{\"u}kboduk studied this problem in \cite{buy14} through the conjectural Rubin–Stark elements by extending and refining the higher rank Euler/Kolyvagin system machinery (that was set up in \cite{buy10}, refining Perrin-Riou’s original approach in \cite{per98}).

 \vspace{1.2mm}

The goal of this article is to incorporate  Nekov\'a{\v r}'s descent formalism developed in \cite{Nek06} with the results of Hsieh \cite{hsi14} to study the $p$-ordinary Iwasawa main conjecture for Hecke characters over the anticyclotomic $\bbZ_p$-power extension of $K$ relying on some standard conjectures.

\subsection*{Notation and Hypotheses} Before we explain our results in greater detail, we set our notation that will be in effect throughout this paper.

\vspace{1.2mm}

 Given a group $G$, let $\mu(G)$ denote the torsion subgroup of $G$. For any number field $M$, we denote by $S_p(M)$ and $S_\infty(M)$ the set of places of $M$ above $p$ and $\infty$, respectively. For a finite set $S \supseteq S_\infty(M)$, we denote by $M_S$ the maximal extension of $M$ unramified outside $S$. Put $G_{M,S} :=\gal(M_S /M)$ and $G_M=\gal(\overline{M}/M)$.

\vspace{1.2mm}

Let $K_\infty$  be the compositum of the cyclotomic $\bbZ_p$-extension $K^{\rm cyc}$ and the anticyclotomic $\bbZ_p^g$-extension $K^{\rm ac}$ of $K$. If we assumed the Leopoldt conjecture for $K$, then $K_\infty$ would be the maximal $\bbZ_p$-power extension of $K$. Note that when $K/\bbQ$ is abelian the Leopoldt conjecture holds. (cf. \cite[Theorem 10.3.6]{neu08}). 

\vspace{1.2mm}

Let $\Gamma_K := \gal(K_\infty/K)$ (resp. $\Gamma_{\rm cyc} := \gal(K^{\rm cyc}/K)$, $\Gamma_{\rm ac} := \gal(K^{\rm ac}/K))$. Then $\Gamma_K $ is a free $\bbZ_p$-module of rank ${g+1}$. Fix a finite extension $\frakF$ of $\bbQ_p$ and denote its ring of integers by $\frakO$. Write $\Lambda(\Gamma_K):=\frakO[[\Gamma_K ]]$ for the Iwasawa algebra, which is isomorphic to the formal power series ring $\frakO[[X_1,\dots, X_{g+1}]]$ in $g + 1$ variables with coefficients in $\frakO$. We define $\Lambda(\Gamma_\cyc)$ and $\Lambda(\Gamma_\ac)$ similarly. 
For any ring $R$ that contains $\frakO$, we also set $\Lambda_R(\Gamma_K):=\Lambda(\Gamma_K)\otimes_\frakO R$.
\vspace{1.2mm}


Let $\chi: G_K \longrightarrow \frakO^\times$ be any non-trivial Dirichlet character whose order is prime to $p$. We let $L_\chi := \overline{K}^{\ker \chi}$ denote the abelian extension of $K$ cut out by $\chi$. Assume that $\chi$ has the property that
\begin{equation} \label{eqn:(1.1)} 
  \chi(\frap)\neq 1 \text{ for any prime } \frap \text{ of } K \text{ lying above } p.
\end{equation}

  and  
\begin{equation} \label{(1.2)}
    \chi \neq \omega, \text{ where } \omega \text{ denotes the Teichm{\" u}ller character}.
\end{equation}
  
 Note that \eqref{eqn:(1.1)} is equivalent to saying that $\frap$ does not split completely in $L_\chi/K$. In this article, we will work with a particular character $\chi$ attached to a CM abelian variety or a Hilbert modular form with CM. 

\vspace{1.2mm}

We fix a $p$-adic CM type $\Sigma$ of $K$, i.e. a subset $\Sigma\subset S_p(K)$ such that $\Sigma\cup \Sigma^c = S_p(K) $ and $\Sigma\cap\Sigma=\emptyset$, where $c$ denote the generator of $\gal(K/F)$. Let $M_\Sigma$ denote the maximal abelian pro-$p$-extension of $L_\infty:=L_\chi K_\infty$ unramified outside $\Sigma$ and let $X_\Sigma:=\gal\left(M_\Sigma/L_\infty\right)$. We put 
\begin{equation*}
    X_\Sigma^\chi:= e_\chi \left(\ccW \otimes_{\bbZ_p} X_\Sigma\right),
\end{equation*}
where $\ccW$ denotes the ring generated by the values of $\chi$ over the $p$-adic completion of the ring of integers of the maximal unramified extension of $\bbQ_p$ and  $e_\chi:=\frac{1}{[L_\chi:K]}\sum_{\sigma\in \gal(L_\chi/K)}\chi(\sigma)\cdot\sigma^{-1}$. By \cite[Theorem 1.2.2]{ht94},  $X_\Sigma^\chi$ is a finitely generated torsion module over $\Lambda(\Gamma_K)$.

\vspace{1.2mm}

 Attached to the $p$-ordinary CM-type $\Sigma$ and the character $\chi$, Katz \cite{kat78} and Hida--Tilouine \cite[Theorem II]{ht93} constructed a $p$-adic $L$-function $\ccL_{p,\Sigma}^{\chi}\in \Lambda_{\ccW}(\Gamma_K)$ that $p$-adically interpolates the algebraic parts (in the sense of \cite{shi75}) of the critical Hecke $L$-values for $\chi$ twisted by the characters of $\Gamma_K$ (see \cite{ht93}). 

 \vspace{1.2mm}

 The following is the Iwasawa Main Conjecture for the CM field $K$ and the character $\chi$ (see \cite{ht94}).

\begin{conj}(Iwasawa Main Conjecture)\label{IMCforCM} ---
    The characteristic ideal of $X_\Sigma^\chi$ is generated by the Katz $p$-adic $L$-function :
\begin{equation*}
    \Char_{\Lambda_{\ccW}(\Gamma_K)}\left( X_\Sigma^\chi \right)=\left(\ccL_{p,\Sigma}^\chi\right).
\end{equation*}
\end{conj}

\begin{remark} ---
    When $K$ is an imaginary quadratic field, Conjecture \ref{IMCforCM} is proved by Rubin in \cite[Theorem 4.1 (i)]{Rub91}.
\end{remark}

In general, Hsieh obtained the following result in \cite{hsi14} towards the validity of Conjecture \ref{IMCforCM}.

\begin{theorem}[Corollary 2, \cite{hsi14}]\label{hsiehstheorem} ---
    Assume that the following hypotheses hold:
    \begin{enumerate}[label=(\roman*)]
    \item $p>5$ is prime to the minus part of the class number of $K$, to the order of $\chi$, and is unramified in $F/\bbQ$.
    
    \item $\chi$ is unramified in $\Sigma^c$ and $\chi\omega^{-a}$ is unramified at $\Sigma$ for some integer $a\not\equiv 2 \mod{p-1}$.

    \item $\chi$ is \textit{anticyclotomic} in the sense that $\chi(c\delta c^{-1})=\chi(\delta)^{-1}$ for $\delta\in\Delta$ and $c\in G_{F}$ that restricts to the generator of $\gal(K/F)$.

    \item $\chi(\frap)\neq 1$ for any $\frap\in\Sigma$ .
    \item The restriction of $\chi$ to $G_{K\left(\sqrt{p^*}\right)}$ (where $p^*=(-1)^{\tfrac{p-1}{2}}p$) is nontrivial.
    
\end{enumerate}
Then Conjecture \ref{IMCforCM} holds.
\end{theorem}

\subsection*{Statements of the results}  Let $f$ be a normalized Hilbert newform of parallel (even) weight $(k,\dots,k)$, level $\fran\subset \ccO_F$ with complex multiplication by $K$. Then there is a Hecke character $\varphi$ of $K$ of infinity type $(k-1)\sum_{\sigma\in \widetilde{\Sigma}}\sigma$ for some CM type $\widetilde{\Sigma}$ of $K$ such that $f$ is the cusp from associated to $\varphi$ (see \cite{hida79}). Fix an embedding $\iota_p : \overline{\bbQ} \xhookrightarrow{} \bbC_p$ . Then $\widetilde{\Sigma}$ and $\iota_p$ give rise to the $p$-adic CM-type $\Sigma$.

\vspace{1.2mm}
 
Let $c$ denote the involution on $G_K$ induced by complex conjugation. As in $\S 2.1$ of \cite{AR07}, we assume further that $\overline{\varphi}\circ c=\varphi$, so that the sign of the functional equation for the Hecke $L$-series attached to $f$ is $\pm1$.

\vspace{1.2mm}

Let $K_f$ denote the number field generated by the Fourier coefficients of $f$. Assume that $p$ does not ramify in $K/\bbQ$ and $K_f/\bbQ$, and $f$ is $p$-ordinary. Let $\widetilde{H}_{f}^2\left(G_{K,S},\bbT_\chi^{\ac,\iota};\Delta_\Sigma\right)$ denote the extended Selmer group attached to the representation $\bbT_\chi^{\ac,\iota}$ (cf. Definition \ref{selmer_complex}), and ${\rm Reg}(\bbT_\chi^{\rm ac,\iota})$ denote the $\Lambda(\Gamma_\ac)$-adic regulator given as in Definition \ref{descend_defn} (vi).  

\vspace{1.2mm}

We write the power series expansion
\begin{equation*}
    \ccL_{p,\Sigma}^\chi:= \ccL_0+\ccL_1\cdot(\gamma_\cyc-1)+\ccL_2\cdot(\gamma_\cyc-1)^2+\dots \in \Lambda_{\ccW}(\Gamma_{\rm ac}) \widehat{\otimes} \Lambda_{\ccW}(\Gamma_{\rm cyc})
\end{equation*}
where $\gamma_{\rm cyc}$ is a fixed topological generator of $\Gamma_{\rm cyc}$, and $\ccL_i$ are elements of $\Lambda_\ccW(\Gamma_\ac)$.

\begin{thmx}\label{maintheoremA} ---
    Assume the hypotheses of Theorem \ref{hsiehstheorem} as well as Hypothesis \ref{hypothesisonfiniteness}. If $\varphi\overline{\varphi}=\bbN^{k-1}$,  ${\rm Reg}(\bbT_\chi^{\rm ac,\iota})\neq 0$ and the sign in the functional equation of $f$ is $-1$, then we have
    \begin{equation*}
    {\rm char}_{\Lambda_{\ccW}(\Gamma_{\rm ac})}\Big(\widetilde{H}_{f}^2\left(G_{K,S},\bbT_\chi^{\ac,\iota};\Delta_\Sigma\right)_{\rm tor} \Big)\cdot{\rm Reg}(\bbT_\chi^{\rm ac,\iota})=\left(\ccL_1\right).
\end{equation*}
\end{thmx}

The next result will be related to the Mordell--Weil ranks of abelian varieties with CM.  Let $L$ be a finite abelian extension of the CM field $K$ with $\Delta:=\gal(L/K)$.  Fix an odd rational prime $p$ that is not ramified in $L/\bbQ$ and not dividing $|\Delta|$. Let $A$ be a polarized simple abelian variety of dimension $g$ defined over $L$ with CM-type $(K,\Sigma)$, i.e. there exists $K\xrightarrow{\;\sim\;} {\rm End }(A)_\bbQ$. Assume that $A$ has good ordinary reduction at $p$

\begin{thmx}\label{maintheoremB} ---
     Let $A/L$ be a simple abelian variety with CM-type $(K,\Sigma)$ such that $K(A_{\rm tor})\subset L$. Assume that the hypotheses of Theorem~\ref{maintheorem} and the non-anomaly condition (see \eqref{non-anomaly_cond}) hold true. If $L(A/L, 1)\neq 0$, then both $A(L)$ and $\Sha(A/L)[\frap^\infty]$ are finite.
\end{thmx}

The only results in this direction were limited to the case when the abelian variety $A$ in question is CM by $K$ defined over the CM field $K$ (cf. $L=K$) or is of ${\rm GL}_2$-type and is defined over $\bbQ$; or else to the case ${\rm dim}(A) = 1$ (that is, $A$ is an elliptic curve).

\subsection*{Organization} We provide a brief overview of the key technical components involved in the proofs of Theorems \ref{maintheoremA} and \ref{maintheoremB}. In Section \ref{selmerstructuresandselmercomplexes}, we review  the Selmer structures introduced by Mazur and Rubin in \cite{MR04} as well as Nekov{\' a}{\v r}’s Selmer complexes \cite{Nek06}.  We then reformulate the Iwasawa main conjecture for the CM field $K$ and the branch character $\chi$ in terms of Selmer complexes, following \cite{bs22}. The section concludes with a summary of Nekov{\'a}{\v r}’s descent formalism from \cite[\S 11]{Nek06}, based on the exposition in \cite[\S 6.3]{BL20}. In Section \ref{heckecharactersattachedtoCMhilbertforms}, we establish several technical results to apply the descent formalism to the Hecke characters attached to CM Hilbert modular forms, leading to the proof of Theorem \ref{maintheoremA} in Section \ref{anticyclotomicmainconjecture}. Section \ref{heckecharofabelianvarieties} focuses on CM abelian varieties defined over an abelian extension of the CM field $K$. After stating the Iwasawa main conjecture for the Hecke character attached to the CM abelian variety in question,  following the approach in \cite[\S 4.4]{buy14} we prove Theorem \ref{maintheoremB}.

\subsection*{Outlook} As in \cite{mai08,ht94,hid09,AH06,AR07,hsi14,buy14}, the $p$-ordinary condition plays also a vital role in this article. One difficulty in the absence of this hypothesis is that neither the Iwasawa modules that one naturally considers are torsion nor the obvious candidates for $p$-adic $L$-function lie in the Iwasawa algebra. To overcome this difficulty, one needs to appeal to $\pm$-Iwasawa theory \cite{kob03,per93, per03, pol03,kato04, PR04}.

\vspace{1.2mm}

Let $E$ be an elliptic curve over $\bbQ$. In \cite{ah05}, Agboola and Howard proved an anticyclotomic Iwasawa main conjecture for $E$ using $\pm$-Selmer groups and elliptic units, assuming Rubin's fundamental conjecture \cite{rub92} (later proved in \cite{bko21}). B{\"u}y{\"u}kboduk \cite{buy15} extended this to newforms of even weight $k \geq 2$ with CM by $K$.

\vspace{1.2mm}

In \cite{buy18}, B{\"u}y{\"u}kboduk proved the cyclotomic main conjecture of $E$ defined over a totally real field, which has CM by an imaginary quadratic field, for a supersingular prime $p$ using the Rubin-Stark elements, conditional on their existence and the truth of reciprocity laws relating them to the appropriate $p$-adic $L$-functions. 

\vspace{1.2mm}

In a forthcoming work, we adapt the approach developed in \cite{BL15,buy18} for the conjectural Rubin–Stark elements along with the new constructions in \cite{bko21} of local points to prove the multi-signed anticyclotomic main conjectures of the Iwasawa theory for CM elliptic curves at \textit{supersingular} primes.

\subsection*{Acknowledgements} I would like to thank my advisor K\^{a}z{\i}m B\"{u}y\"{u}kboduk who suggested this project. I am also grateful for his continuous guidance and enlightening discussions during the preparation of this paper. I also would like to thank Antonio Lei, Christian Wuthrich, F{\i}rt{\i}na K{\"u}\c{c}{\"u}k and Daniele Casazza for their comments and suggestions on this article.

\section{Selmer structures and Nekov{\' a}{\v r}’s Selmer Complexes}\label{selmerstructuresandselmercomplexes}

For any non-archimedean prime $\lambda$ of $K$, fix a decomposition group $D_\lambda$ and the inertia subgroup $I_\lambda \subset D_\lambda$. Let $(-)^\vee := \Hom(-,\bbQ_p/\bbZ_p)$ denote Pontryagin duality functor. Observe that $(-)^\vee \otimes\frakO = \Hom(-,\frakF/\frakO)$. Considering this relation, we will write $X^\vee$ for $\Hom(X, \frakF/\frakO)$ when $X$ is an $\frakO$-module. 

\vspace{1.2mm}

 Suppose $\varphi:G_K\longrightarrow \frakO^\times$ is a continuous $p$-adic Hecke character which is Hodge-Tate. Write $\frakO^\times=\mu(\frakF^\times)\times U^{(1)}$ and $\langle \varphi \rangle:G_K \xtwoheadrightarrow{\;} U^{(1)} $ is the map $\varphi$ followed by the projection $\frakO^\times \xtwoheadrightarrow{\;}  U^{(1)}$. Set
\begin{equation*}
    \chi:= \omega_\varphi=\varphi\cdot\langle\varphi\rangle^{-1}: G_K \xtwoheadrightarrow{} \mu(\frakF^\times)\xhookrightarrow{} \frakO^\times
\end{equation*}
and assume that
\begin{equation}\label{(1.3)}
    \chi \text{ is not the trivial character.}
\end{equation}
   
Define the $G_K$-representation $T_\chi:=\frakO(1)\otimes\chi^{-1}$ that is equal to $\frakO$ as an $\frakO$-module and on which $G_K$ acts via $\chi_\cyc\chi^{-1}$, where $\chi_\cyc$ denotes the cyclotomic character of $K$. Let $\varphi^*=\varphi^{-1}\otimes\chi_\cyc$ and $T:=\frakO(\varphi^*)$ be the free $\frakO$-module of rank one on which $G_K$ acts via the Hecke character $\varphi^*$. Observe that $T=T_\chi\otimes\langle\varphi\rangle^{-1}$.

\subsection{Selmer Structures}

In this section, we will recall Mazur and Rubin's definition of \textit{Selmer structures} (cf. \cite[\S 2]{MR04}). Let $R$ be a complete local Noetherian $\frakO$-algebra, and let $X$ be an $R[[G_K]]$-module which is free of rank $d$ over $R$.  We will be interested in the case where $X$ is $T_\chi\otimes \Lambda(\Gamma_K)$ or its relevant quotients by an ideal of $\Lambda(\Gamma_K)$ such as the augmentation ideal. 

\vspace{1.2mm}

\begin{definition}--- A Selmer structure $\ccF$ on $X$ is a collection of the following data:
\begin{enumerate}[label=(\roman*)]
    \item A finite set $\Sigma(\ccF)$ of places of $K$, including all infinite places and primes lying above $p$ and all primes at which $X$ is ramified.
    \item For every $v\in \Sigma(\ccF)$, a local condition on $X$ viewed as an $R[[D_v]]$-module; namely, a choice of an $R$-submodule
    $$H_\ccF^1(K_v,X)\subset H^1(K_v,X). $$
\end{enumerate}
If $v\not\in \Sigma(\ccF)$ we will also write $H_\ccF^1(K_v,X)=H_f^1(K_v,X)$, where the module $H_f^1(K_v,X)$ is the finite part of $H^1(K_v,X)$ in the sense of \cite[Definition~1.1.6]{MR04}.
\end{definition}

\vspace{1.2mm}

For a Selmer structure $\ccF$ on $X$, set $\Sigma(\ccF^*)=\Sigma(\ccF)$, and for $v\in \Sigma(\ccF^*)$ define $H_{\ccF^*}^1(K_v,X^\vee(1)):=H_\ccF^1(K_v,X)^\perp$ as the orthogonal complement of $H_{\ccF}^1(K_v,X)$ with respect to the local Tate pairing. The Selmer structure $\ccF^*$ on $X^\vee(1)$ will be called the dual Selmer structure.

\vspace{1.2mm}

\begin{definition}\label{selmer_module} ---
If $\ccF$ is a Selmer structure on $X$, we define the Selmer module $H_\ccF^1(K,X)$ as the kernel of the sum of the localization maps
\begin{equation*}
    H_\ccF^1(K,X):={\rm ker}\bigg(H^1(G_{K,\Sigma(\ccF)},X)\longrightarrow \bigoplus_{v\in \Sigma(\ccF)}H^1(K_v,X)\big/H_\ccF^1(K_v,X) \bigg),
\end{equation*}
where $K_{\Sigma(\ccF)}$ denotes the maximal extension of $K$ which is unramified outside $\Sigma(\ccF)$, and $G_{K,\Sigma(\ccF)}:={\rm Gal}(K_{\Sigma(\ccF)}/K)$. We also define the dual Selmer module $H^1_{\ccF^*}(K,X^\vee(1))$ in a similar way.
\end{definition}

\vspace{1.2mm}

\begin{example}\label{selmerstructureexample} --- We have the following examples (see \cite[Definitions 3.2.1, 5.3.2]{MR04} and \cite[Example 2.12]{buy14}):
\begin{enumerate}[label=(\roman*)]
    \item Assume that $R=\frakO$ and that $X$ is a free $\frakO$-module endowed with a continuous action of $G_K$, which is unramified outside a finite set of places of $K$. We define a Selmer structure $\ccF_{\rm can}$ by setting 
    $$\Sigma(\ccF_{\rm can}):=\{v: X \; \text{\rm is ramified at }v \}\cup \{\frap\mid p\}\cup \{v\mid \infty\}.$$
    \begin{itemize}
        \item if $v \in \Sigma(\ccF_{\rm can})$ with $v\nmid p\infty$, we define
        $$H_{\ccF_{\rm can}}^1(K_v,X):={\rm ker}(H^1(K_v,X)\longrightarrow H^1(K_v^{\rm unr},X\otimes_\frakO K)), $$
where $K_v^{\rm unr}$ denotes the maximal unramified extension of $K_v$.
        \item If $\frap \mid  p$, we define the local condition as
        $$H_{\ccF_{\rm can}}^1(K_\frap,X):=H^1(K_\frap,X) .$$
         \item If $v\mid \infty$, we define the local condition as
        $$H_{\ccF_{\rm can}}^1(K_v,X):=H^1(K_v,X) .$$
    \end{itemize}
The Selmer structure $\ccF_{\rm can}$ is called the canonical Selmer structure on $X$.

\item Assume now that $R$ denotes $\Lambda(\Gamma_K)$ or its anticyclotomic quotient $\Lambda(\Gamma_{\rm ac})$ and $\bbX$ is a free $R$-module on which $G_K$ acts continuously, and $\bbX$ is unramified outside a finite set of places of $K$. We define a Selmer structure $\ccF_R$ on $\bbX$ by setting
$$\Sigma(\ccF_R):=\{v: \bbX \; \text{\rm is ramified at }v \}\cup \{\frap\mid p\}\cup \{v\mid \infty\},$$
and $H_{\ccF_R}^1(K_v,\bbX):=H^1(K_v,\bbX)$ for every $v\in \Sigma(\ccF_R)$. The Selmer structure $\ccF_R$ on $\bbX$ is called the canonical $R$-adic Selmer structure on $\bbX$.
\end{enumerate}
\end{example}

\subsubsection{Modified Selmer structures} Write $\bbT_\chi$ (resp. $\bbT$) for the $G_K$-representation $T_\chi\otimes\Lambda(\Gamma_K)$ (resp. $T\otimes\Lambda(\Gamma_K)$).  Note that the action of $G_K$ on $\bbT_\chi$ is given by its action of both factors and that $G_K$ acts on $\Lambda(\Gamma_K)$ via the tautological character $G_K \twoheadrightarrow \Gamma_K \hookrightarrow \Lambda(\Gamma_K)^\times$.  

\vspace{1.2mm}

\begin{definition}\label{local_cond_selm_st} ---
We define the $\ccF_\Sigma$-modified Selmer structure on $\bbT_\chi$ as follows: 
 \begin{itemize}
        \item $\Sigma_K=\Sigma(\ccF_{\Lambda(\Gamma_K)}):=\{v:\bbT_\chi \; \text{\rm is ramified at }v \}\cup \{\frap\mid p\}\cup \{v\mid \infty\},$
        \item if $v\nmid p$, define $H^1_{\ccF_\Sigma}(K_v,\bbT_\chi):=H^1_{\ccF_\can}(K_v,\bbT_\chi)$
        \item $H^1_{\ccF_\Sigma}(K_p,\bbT_\chi):=\bigoplus_{\fraq\in\Sigma^c}H^1(K_{\fraq},\bbT_\chi)  \subset H^1(K_p,\bbT_\chi):=\bigoplus_{\fraq\mid  p}^sH^1(K_{\fraq},\bbT_\chi).$
    \end{itemize}
\end{definition}

\begin{lemma}\label{twistlem} ---
We have the following twisting isomorphisms
\begin{equation*}
    H^1(M,\bbT_\chi)\otimes\langle\varphi \rangle^{-1}\xrightarrow[]{\;\;\sim\;\;} H^1(M,\bbT),
\end{equation*}
for $M=K,K_p$.
\end{lemma}

\begin{proof}

The result is immediate from \cite[Chapter VI, Proposition 2.1]{Ru00}, as we have $T=T_\chi\otimes\langle\varphi\rangle^{-1}$ and $\langle \varphi\rangle^{-1}$ is a continuous character of $\Gamma_K$.
\end{proof}

\begin{definition} ---
Let $H_{\ccF_\Sigma}^1(K_p,\bbT)\subset H^1(K_p,\bbT)$ be the isomorphic image of $H_{\ccF_\Sigma}^1(K_p,\bbT_\chi)$ under the isomorphism given in Lemma~\ref{twistlem}.
\end{definition}

\begin{lemma}\label{isomorhismofmodifiedselmermodules} ---
We have the following isomorphism of the Selmer groups induced by the isomorphism in Lemma~\ref{twistlem}
\begin{equation*}
    H^1_{\ccF_\Sigma}(K,\bbT_\chi)\otimes\langle\varphi \rangle^{-1}\xrightarrow[]{\;\;\sim\; 
 \;} H^1_{\ccF_\Sigma}(K,\bbT).
\end{equation*}

\end{lemma}
\begin{proof}
   This is \cite[Lemma 2.15]{buy14}.
\end{proof}

\vspace{1.2mm}

 \subsection{Selmer complexes}

 For $\Gamma\in\{\Gamma_K,\Gamma_{\rm cyc}, \Gamma_{\rm ac} \}$ we write $\Lambda(\Gamma)^{\iota}$ for the free $\Lambda(\Gamma)$-module of rank one equipped with the action of $G_K$ given as the composition map
\begin{equation*}
    G_K  \twoheadrightarrow  \Gamma  \xrightarrow{\iota\;:\;g\;\mapsto\; g^{-1}} \Gamma \hookrightarrow \Lambda(\Gamma)^{\times}.
\end{equation*}

Let us define the following $G_K$-representations 
\begin{equation*}
    \mathbb T_\chi^\iota:= T_\chi\otimes \Lambda(\Gamma_K)^{\iota},\;\mathbb T_\chi^{\rm ac,\iota}:=T_\chi\otimes\Lambda\left(\Gamma_{\rm ac}\right)^{\iota} \text{ and } \mathbb T_\chi^{\rm cyc,\iota}:=T_\chi\otimes\Lambda(\Gamma_{\rm cyc})^{\iota}.
\end{equation*}

 For simplicity, let $S:=\Sigma(\ccF_\can)$. For $X\in\{\mathbb T_\chi^\iota,\mathbb T_\chi^{\ac,\iota}, \mathbb T_\chi^{\cyc,\iota} \}$ we set the Greenberg's local conditions:
\[ U^+_v(X):= \begin{cases} 
      C^\bullet(G_v,X),& \text{if } v\in\Sigma^c \\
      0,& \text{if } v\in \Sigma \\
      C^\bullet(G_v/I_v,X^{I_v}),& \text{if } v\not\in S_p(K).
       \end{cases}
\]

Here, for a pro-finite group $G$ and a topological $G$-module $X$, let $C^\bullet(G, X )$ denote the complex of continuous cochains. For each prime $v$ of $K$, we have a canonical injection
\begin{equation*}
    i_v^+:U_v^+(X)\longrightarrow C^{\bullet}(G_v,X).
\end{equation*}

Additionally,  we set 
\begin{equation*}
    U_v^-(X):={\rm Cone}\left(U_v^+(X)\xrightarrow{\;\;-i^+_v\;\;} C^\bullet(G_v,X)\right)
\end{equation*}
 and
 \begin{align*}
    U_{S}^{\pm}(X):=\bigoplus_{v\in S_f}U_v^{\pm}(X); \;\;    i_S^{+}(X)=\bigoplus_{v\in S_f}i_v^+(X).
 \end{align*}

 We also define
$${\rm res}_{S}: C^\bullet(G_{K,S},X)\longrightarrow \bigoplus_{v\in S_f} C^\bullet(G_v,X)$$
as the canonical restriction morphism.

\begin{definition}\label{selmer_complex} ---
The Selmer complex on $X$ associated with a choice of local conditions $\{U_v^+(X)\}_{v\in S}$ on $X$ is given by the complex
\begin{equation*}
    \widetilde{C}_{f}^\bullet(G_{K,S},X;\Delta_\Sigma) := {\rm Cone}\left(C^\bullet(G_{K,S},X)\oplus U_{S}^+(X) \xrightarrow{{\rm res}_{S}-i_S^+}   \bigoplus_{v\in S_f} C^\bullet(G_v,X)\right)[-1],
\end{equation*}
where $[n]$ denotes the shift by $n$.
\end{definition}

We will denote by $\widetilde{R\Gamma}_{ f}(G_{K,S},X;\Delta_\Sigma)$ the  object in the derived category corresponding to the complex $\widetilde{C}_{ f}^\bullet(G_{K,S},X;\Delta_\Sigma)$ and by $\widetilde{H}^i_{f}(G_{K,S},X;\Delta_\Sigma)$ its cohomology.

\begin{lemma}\label{isom_selmer_groups} ---
We have the following isomorphism of $\Lambda(\Gamma_K)$-modules:
\begin{equation*}
        \widetilde{H}^2_{f}(G_{K,S},\bbT_\chi;\Delta_\Sigma)\cong   \left( H_{\ccF_\Sigma^*}^1\left(K,\bbT_\chi^\vee(1)\right)^\iota\right)^{\vee}.
 \end{equation*}
\end{lemma}
\begin{proof}
  Let ${\rm Sel}^{\rm str}(K,\bbT_\chi^\vee(1))$ denote the strict Selmer group defined as in \cite[9.9.6.1]{Nek06}. Comparing the local conditions on ${\rm Sel}^{\rm str}(K,\bbT_\chi^\vee(1))$ and on $ H_{\ccF_\Sigma^*}^1\left(K,\bbT_\chi^\vee(1)\right)$ one can deduce that $H^1_{\ccF_\Sigma^*}(K,\bbT_\chi^\vee(1))= {\rm Sel}^{\rm str}(K,\bbT_\chi^\vee(1))$. On the other hand, it follows from \cite[8.9.6.2]{Nek06} that we have the following isomorphism
\begin{equation*}
      \widetilde{H}^2_{f}(G_{K,S},\bbT_\chi;\Delta_\Sigma)\xrightarrow{\;\;\sim\;\;} \left( \widetilde{H}_{f}^1\left(G_{K,S},\bbT_\chi^\vee(1);\Delta_\Sigma\right)^\iota\right)^{\vee} .
 \end{equation*}
 Since $\chi(\frap)\neq 1$ holds for $\chi$, the assertion follows from \cite[Lemma 9.6.3]{Nek06}.
\end{proof}

In \cite[Appendix C]{bs22}, B{\"u}y{\"u}kboduk and Sakamoto reformulated Conjecture \ref{IMCforCM} in terms of Selmer complexes.

 \begin{lemma}\label{reformulationofIMClemma} --- The $\Lambda_\ccW(\Gamma_K)$-module $\widetilde{H}^2_{f}(G_{K,S},\bbT_\chi^\iota;\Delta_\Sigma)$ is torsion, and we have
\begin{align*}
    \Char_{\Lambda_\ccW(\Gamma_K)}\left( \widetilde{H}^2_{f}(G_{K,S},\bbT_\chi^\iota;\Delta_\Sigma) \right) &= \Char_{\Lambda_\ccW(\Gamma_K)}\left( \left( H_{\ccF_\Sigma^*}^1\left(K,\bbT_\chi^\vee(1)\right)^\iota\right)^{\vee}\right) \\ &= \Char_{\Lambda_\ccW(\Gamma_K)}\left( X_\Sigma^\chi \right).
\end{align*}     
 \end{lemma}
 \begin{proof}
     The assertion follows from Lemma \ref{isom_selmer_groups} and \cite[Lemma C.7]{bs22}.
 \end{proof}

\subsection{Nekov{\' a}{\v r}’s descent formalism} Following the summary in \cite[\S6.3]{BL20}, we will recall the descent formalism developed in \cite[\S11]{Nek06}. Consider any complete local Noetherian ring $R$ and a free $R$-module $X$ of finite rank endowed with a continuous action of $G_{K,S}$.

\vspace{1.2mm}

\begin{definition}\label{descend_defn} ---
\begin{enumerate}[label=(\roman*)]
    \item Let us set $R_{\rm cyc}:=R \widehat{\otimes}\Lambda(\Gamma_{\rm cyc})$ and $X_{\rm cyc}:=X \widehat{\otimes}\Lambda(\Gamma_{\rm cyc})^{\iota}$. A choice of Greenberg local conditions $\Delta(X)$ as in \cite[\S6.1]{Nek06} gives rise to local conditions for $\Delta(X_{\rm cyc})$ for $X_{\rm cyc}$ and allows one to consider the Selmer complex
$$\widetilde{\bfR\Gamma}_{f}(G_{K,S},X_{\rm cyc};\Delta(X_{\rm cyc})) \in {{\rm D_{ft}}({}_{R_{\cyc}}{\rm Mod})}. $$

\item Let ${\rm ht}_1(R)$ denote the set of height-$1$ primes of $R$. For any $P\in {\rm ht}_1(R)$ of $R$, we write $P_{\rm cyc}\subset R_{\rm cyc}$ for the prime ideal of $R_{\rm cyc}$ generated by $P$ and $\gamma_{\rm cyc}-1$, where $\gamma_{\rm cyc}$ is a fixed topological generator of $\Gamma_{\rm cyc}$.

\item Suppose that $M$ is a torsion $R_{\rm cyc}$-module of finite-type. Following  \cite[\S11.6.6]{Nek06}, we define
$$a_{P}(M_{P_{\rm cyc}}):={\rm length}_{R_P}\bigg(\varinjlim_r\dfrac{(\gamma_{\rm cyc}-1)^{r-1}M_{P_{\rm cyc}}}{(\gamma_{\rm cyc}-1)^{r}M_{P_{\rm cyc}}} \bigg). $$

\item For each $P\in {\rm ht}_1(R)$, let ${\rm Tam}_v(X,P)$ denote the Tamagava factor at a prime $v\in S_f$ prime to $p$ (see \cite[Definition 7.6.10]{Nek06}).

\item Set the linear dual $X^*:={\rm Hom}(X,R)(1)$, which is a free $R$-module of finite rank endowed with a continuous action of $G_{K,S}$ in the obvious manner. Let $\Delta(X^*(1))$ denote the dual local condition for $X^*(1)$ (in the sense that \cite[\S10.3.1]{Nek06} applies). We then define an $R$-adic (cyclotomic) height pairing as
$$\frah_{X}^{\rm Nek}: \widetilde{H}_{f}^1(G_{K,S},X;\Delta(X))\otimes \widetilde{H}_{f}^1(G_{K,S},X^*(1);\Delta(X^*(1))) \longrightarrow R$$
given as \cite[\S11.1.4]{Nek06} with $\Gamma=\Gamma_{\rm cyc}$.

\item If $R$ is a Krull domain, we define the $R$-adic regulator ${\rm Reg}(X)$ on the setting
\begin{equation*}
    {\rm Reg}(X):={\rm char}_R \bigg({\rm coker}\big(\widetilde{H}_{f}^1(G_{K,S},X;\Delta(X))\xrightarrow{{\;\rm adj}(\frah_X^{\rm Nek})\;}\widetilde{H}_{f}^1\left(G_{K,S},X^*(1);\Delta\left(X^*(1)\right)\right) \big)  \bigg)
\end{equation*}
where ${\rm adj}$ denotes the adjunction. Note that ${\rm Reg}(X)$ is non-zero if and only if $\frah_X^{\rm Nek}$ is non-degenarete.
\end{enumerate}
\end{definition}

\vspace{1.2mm}

\begin{definition}---
\begin{enumerate}[label=(\roman*)]
    \item Let $R$ be a Krull domain so that $R_P$ is a discrete valuation ring for every $P\in {\rm ht}_1(R)$ and $(R_{\rm cyc})_{P_{\rm cyc}}$ is a regular ring. Let $X$ be a finitely generated torsion $R_{\rm cyc}$-module, and let $r(X_{P_{\rm cyc}})$ denote the largest integer such that 
    \begin{equation*}
        {\rm char}_{(R_{\rm cyc})_{P_{\rm cyc}}}(X_{P_{\rm cyc}})\in (\gamma_{\rm cyc}-1)^{r(X_{P_{\rm cyc}})}(R_{P_{\rm cyc}}).
    \end{equation*}
We set 
\begin{equation*}
    \partial_{\rm cyc}{\rm char}_{(R_{P_{\rm cyc}})}(X_{P_{\rm cyc}}):=\bigg( \mathbbm{1}_{\rm cyc}(\gamma_{\rm cyc}-1)^{-r(X_{P_{\rm cyc}})}(M_{P_{\rm cyc}})\bigg)\subset R_P,
\end{equation*}
where $\mathbbm{1}_{\rm cyc}:R_{\rm cyc}\longrightarrow R$ is the augmentation map. 
\item Suppose further that every prime ideals in ${\rm ht}_1(R)$ is principal; this is equivalent to saying that $R$ is a unique factorization domain. In this case, one may define an integer $r(M)$ with
$${\rm char}_{R_{\rm cyc}}(M)\in (\gamma_{\rm cyc}-1)^{r(M)}R_{\rm cyc} $$
and set 
\begin{equation*}
    \partial_{\rm cyc}{\rm char}_{(R_{\rm cyc})}(M):=\mathbbm{1}_{\rm cyc}\Big((\gamma_{\rm cyc}-1)^{-r(M)}{\rm char}_{R_{\rm cyc}}(M) \Big) \in R.
\end{equation*}
In particular, we have 
\begin{equation}\label{apequalchar}
      {\rm char}_{(R_{\rm cyc})_{P_{\rm cyc}}}(M_{P_{\rm cyc}}) =\left( {\rm char}_{(R_{\rm cyc})}(M)\right)_{P_{\rm cyc}},
\end{equation}
\begin{equation}\label{apequalchar2}
    \partial_{\rm cyc}{\rm char}_{(R_{\rm cyc})_{P_{\rm cyc}}}(M_{P_{\rm cyc}})=\big( \partial_{\rm cyc}{\rm char}_{R_{\rm cyc}}(M)\big)_P.
\end{equation}
\end{enumerate}
\end{definition}

\vspace{1.2mm}

\begin{lemma}[{\rm Nekov\'a\v r}]\label{nekovardescentlemma} --- Suppose that $R$ is a Krull domain. Then we have
$$a_P(M_{\rm cyc})={\rm ord}_P\big(\partial_{\rm cyc} {\rm char}_{(R_{\rm cyc})_{P_{\rm cyc}}}(M_{P_{\rm cyc}}) \big). $$
Furthermore, if $R$ is a unique factorization domain, then 
$$a_P(M_{\rm cyc})=\big(\partial_{\rm cyc}{\rm char}_{R_{\rm cyc}}(M)\big)_P.$$
\end{lemma}
\begin{proof}
    The first part is \cite[Lemma 11.6.8]{Nek06}. The second part follows from this combined with \eqref{apequalchar2}.
\end{proof}

\vspace{1.2mm}

\begin{prop}[{\rm Nekov\'a\v r}]\label{abstract_BSD_formulaeprop} --- With the definition given in \ref{descend_defn}, assume that the ring $R$ is a Krull integral domain and $\widetilde{H}_{f}^1(G_{K,S},X_{\rm cyc};\Delta(X_{\rm cyc}))$ and $\widetilde{H}_{f}^2(G_{K,S},X_{\rm cyc};\Delta(X_{\rm cyc}))$ are torsion $R_{\rm cyc }$-modules. Suppose further that
\begin{align*}
    \widetilde{H}_{f}^0(G_{K,S},X;\Delta(X))&=\widetilde{H}_{f}^3(G_{K,S},X;\Delta(X))=0\\
    \widetilde{H}_{f}^0(G_{K,S},X_{\rm cyc};\Delta(X_{\rm cyc}))&=\widetilde{H}_{f}^3(G_{K,S},X_{\rm cyc};\Delta(X_{\rm cyc}))=0
\end{align*}
and that the $R$-adic height pairing $\frah_X^{\rm Nek}$ is non-degenerate (namely, ${\rm Reg}(X)$ is non-zero). Let $P\in {\rm ht}_1(R)$ be any height-one prime ideal such that ${\rm Tam}_v(X,P)=0$ for every $v\in S_f$ prime to $p$. Then
\begin{align*}
    {\rm ord}_P\bigg(\partial_{\rm cyc}{\rm char}_{(R_{\rm cyc})_{P_{\rm cyc}}}\big(\widetilde{H}_{f}^2(G_{K,S},X_{\rm cyc}&;\Delta(X_{\rm cyc}) \big)_{P_{\rm cyc}}\bigg)={\rm ord}_P({\rm Reg}(X))\\
    &+{\rm length}_{R_{P}}\bigg(\big(\widetilde{H}_{f}^2(G_{K,S},X;\Delta(X))_P \big)_{R_P-{\rm tor}} \bigg).
\end{align*}
If in addition, $R$ is a Krull unique factorization domain, then we have
\begin{align*}
    {\rm ord}_P\bigg(\partial_{\rm cyc}{\rm char}_{(R_{\rm cyc})_{P_{\rm cyc}}}\big(\widetilde{H}_{f}^2(G_{K,S},X_{\rm cyc}&;\Delta(X_{\rm cyc}) \big)_{P_{\rm cyc}}\bigg)={\rm ord}_P({\rm Reg}(X))\\
    &+{\rm ord}_{R_{P}}\bigg(\big(\widetilde{H}_{f}^2(G_{K,S},X;\Delta(X))_P \big)_{R_P-{\rm tor}} \bigg).
\end{align*}
\end{prop}
\begin{proof}
    The first assertion is \cite[11.7.11]{Nek06}, combined with Lemma~\ref{nekovardescentlemma}. The second statement follows from the first, using \eqref{apequalchar} and \eqref{apequalchar2}.
\end{proof}

\section{Hecke characters attached to CM Hilbert modular forms}\label{heckecharactersattachedtoCMhilbertforms}

In this section, we explore a consequence of Theorem~\ref{hsiehstheorem} towards the anticyclotomic Iwasawa main conjecture for CM Hilbert modular forms (as a generalization of the main results in \cite{AH06, AR07}).

\vspace{1.2mm}

  Let $f$ be a normalized Hilbert newform of parallel (even) weight $(k,\dots,k)$, level $\fran\subset \ccO_F$ with CM by $K$ and fix a rational prime $p\geq 5$ at which $f$ is ordinary. Assume that $p$ does not ramify in $F/\bbQ$ and $K_f/\bbQ$, where $K_f$ denotes the Hecke field generated by the Hecke eigenvalues $c_f(\fraa)$ of $f$ for all $\fraa\subset \ccO_F$. Fix a prime ideal $\frakP$ of $\ccO_{K_f}$ lying above $p$, and let 
\begin{equation*}
    \rho_f: G_F \longrightarrow {\rm GL}_2(K_{f,\frakP})
\end{equation*}
denote the Galois representation associated to $f$. Since $f$ is $p$-ordinary, for all primes $\wp$ of $F$ above $p$, we have $\wp\ccO_K=\frap\frap^*$ with $\frap\neq \frap^*$. 

\subsection{$p$-adic Hecke characters}\label{padicheckechar}
 
 By the theory of complex multiplication, there is a Hecke character $\varphi$ of $K$ associated to $f$ whose infinity type is $(k-1)\sum_{\sigma\in \Sigma}\sigma$ for some CM type $\Sigma$ of $K$ (see \cite{hida79}). This implies that $\rho_f|_{G_K}\cong \varphi_\frap \oplus \varphi_\frap^c $ for the $\frap$-adic avatar $\varphi_\frap$ of $\varphi$, where $\frap$ is a prime of $K$ lying above $\wp$ and $c$ denotes the involution of $G_K$ induced by the complex conjugation.

 \vspace{1.2mm}

Assume also that $\overline{\varphi}\circ c=\varphi$ (i.e. $\varphi^c_\frap=\varphi_{\frap^*}$) so that the sign $W(\varphi)\in \{\pm 1\}$ of the functional equation
of $\varphi$ (namely, the sign of the functional equation for the Hecke $L$-series attached to $f$) makes sense. Then we have $\fraf_\varphi=\overline{\fraf}_\varphi$, where $\fraf_\varphi$ denotes the conductor of $\varphi$.

 \vspace{1.2mm}
 
 Denote by $\frakO$ the integer ring of $\frakF:=K_{f,\frakP}$. Consider the self-dual twist $\psi:=\varphi_\frap\otimes \chi_\cyc^{1-k/2}$ of $\varphi_\frap$.  Write $\frakO^\times=\mu(\frakF^\times)\times U^{(1)}$ and $\langle \varphi \rangle:G_K\twoheadrightarrow U^{(1)} $ is the map $\varphi$ followed by the projection $\frakO^\times \twoheadrightarrow U^{(1)}$. Set
\begin{equation*}
    \chi:= \omega_\varphi=\varphi\cdot\langle\varphi\rangle^{-1}: G_K \longrightarrow \mu(\frakF^\times)\hookrightarrow \frakO^\times
\end{equation*}
and assume that $\omega_\psi$ is not trivial. Define the $G_K$-representation $T_\chi:=\frakO(1)\otimes \chi^{-1}$ as in \S \ref{padicheckechar}.

\subsection{Application of Nekov\'a\v r's descent formalism}
 We will need several technical results to apply Nekov\'a\v r's formalism.

\begin{lemma}\label{euler_characteristic} ---
For $(X,\Gamma)=(\bbT_\chi^\iota,\Gamma_K)$ or $(\bbT_\chi^{\ac,\iota},\Gamma_\ac)$, let $\widetilde{h}^i$ denote the rank of the $\Lambda(\Gamma)$-modules $ \widetilde{H}^2_{f}(G_{K,S},X;\Delta_\Sigma)$. Then we have
\begin{equation*}
    \widetilde{h}^0-\widetilde{h}^1+\widetilde{h}^2-\widetilde{h}^3=0.
\end{equation*}
\end{lemma}
\begin{proof}
Since $K$ is a CM field, it follows from the argument in the proof of \cite[Proposition 9.7.2 (ii)]{Nek06} applied to our case and \cite[Theorem 8.9.15]{Nek06} that
\begin{align*}
\widetilde{h}_0-\widetilde{h}_1+\widetilde{h}_2-\widetilde{h}_3&=\sum_{v\;\mid \; \infty}1-\sum_{v\in \Sigma^c}[K_v:\bbQ_p]\\
&=[F:\bbQ]-\sum_{v\in \Sigma^c}[K_v:\bbQ_p].
\end{align*}
It follows from our running hypothesis \eqref{eqn:ordinaryassumption} that
\begin{equation*}
 [F:\bbQ]=  \sum_{v\in \Sigma^c}[K_v:\bbQ_p],
\end{equation*}
which completes the proof.
\end{proof}

\vspace{1.2mm}

\begin{lemma}\label{tamagawa_ac} --- 
Let $\fraq$ be a height-one prime ideal of $\Lambda(\Gamma_{\rm ac})$. Then
 ${\rm Tam}_v(\bbT_\chi^{\rm ac,\iota},\fraq)=0$ for every $v\in S_f$ prime to $p$.
\end{lemma}
\begin{proof} The result follows from \cite[Corollary 8.9.7.4 and 7.6.10.8] {Nek06} applied with $X=\bbT_\chi^{\rm ac,\iota}$ and $\overline{R}=\Lambda(\Gamma_{\rm ac})$.

\end{proof}

\begin{lemma}\label{rank_zero_lemma} ---
We have the following vanishing results:
 \begin{align*}
    \widetilde{H}_{f}^0\left(G_{K,S},\bbT_\chi^{\ac,\iota};\Delta_\Sigma\right)&=\widetilde{H}_{f}^3\left(G_{K,S},\bbT_\chi^{\ac,\iota};\Delta_\Sigma\right)=0\\
    \widetilde{H}_{f}^0\left(G_{K,S},\bbT_\chi^\iota;\Delta_\Sigma\right)&=\widetilde{H}_{f}^3\left(G_{K,S},\bbT_\chi^\iota;\Delta_\Sigma\right)=0 .
\end{align*}
\end{lemma}
\begin{proof}
 The verification of $\widetilde{H}_{f}^0\left(G_{K,S},\bbT_\chi^\iota;\Delta_\Sigma\right)=\widetilde{H}_{f}^0\left(G_{K,S},\bbT_\chi^{\ac,\iota};\Delta_\Sigma\right)=0$ follows from \cite[Proposition~9.7.2]{Nek06}. Since \eqref{eqn:(1.1)} and \eqref{(1.2)} hold true, we deduce that (see \cite[Section 9.5]{Nek06} and \cite[Page 5859]{buy14}) $\widetilde{H}_{f}^3(G_{K,S},T_\chi;\Delta_\Sigma)=\left(\frakO\left[\Delta_\chi/\Delta_\frap\right]\right)^\chi=0$, where $\Delta_\chi:=\gal(L_\chi/K)$ and $\Delta_\frap$ is the decomposition group of a fixed prime ideal of $L_\chi$ lying above $\frap$ in $\Delta_\chi$. The assertion $\widetilde{H}_{f}^3\left(G_{K,S},\bbT_\chi^\iota;\Delta_\Sigma\right)=\widetilde{H}_{f}^3\left(G_{K,S},\bbT_\chi^{\ac,\iota};\Delta_\Sigma\right)=0$ now follows from \cite[Propositions 8.10.13]{Nek06}. 
\end{proof}

\vspace{1.2mm}

Let $H_{\ccF_{\ccL}^*}^1\left(K,T_\chi^\vee(1)\right)$ denote the Selmer module given by the $\ccL$-modified Selmer structures in \cite[Definition 2.14]{buy14}.

\begin{hyp}\label{hypothesisonfiniteness}
    We assume that $H_{\ccF_{\ccL}^*}^1\left(K,T_\chi^\vee(1)\right)$ is finite.
\end{hyp}

\begin{remark} ---
    If we assume the Rubin--Stark conjecture, it then follows from \cite[Theorem 4.1]{buy14} that  Hypothesis \ref{hypothesisonfiniteness} holds true.
\end{remark}

\begin{prop}\label{prop_rank} ---
    For the $\Lambda(\Gamma_\ac)$-module $\widetilde{H}^1_{f}\left(G_{K,S},\bbT_\chi^{\ac,\iota};\Delta_\Sigma\right)$ we have 
    \begin{equation*}
        {\rm rank}_{\Lambda(\Gamma_{\rm ac})}\left(\widetilde{H}^1_{f}\left(G_{K,S},\bbT_\chi^{\ac,\iota};\Delta_\Sigma\right)\right)\leq 1.
    \end{equation*}
\end{prop}
\begin{proof}
    First, observe that $\widetilde{H}^1_{f}\left(G_{K,S},\bbT_\chi^{\ac,\iota};\Delta_\Sigma\right)$ and $\widetilde{H}^2_{f}\left(G_{K,S},\bbT_\chi^{\ac,\iota};\Delta_\Sigma\right)$ have the same $\Lambda(\Gamma_\ac)$-rank by Lemma~\ref{euler_characteristic} and Lemma \ref{rank_zero_lemma}. By comparing the local conditions and utilizing \cite[8.9.10]{Nek06} one can show that
    \begin{equation*}
          \widetilde{H}^2_{f}\left(G_{K,S},\bbT_\chi^{\ac,\iota};\Delta_\Sigma\right) \cong  \left( H_{\ccF_\Sigma^*}^1\left(K,(\bbT_\chi^\ac)^\vee(1)\right)^\iota\right)^{\vee}.
    \end{equation*}
By the exact sequence in \cite[Proposition 2.21 (i)]{buy14}, it suffices to show that the  $\Lambda(\Gamma_\ac)$-module $\left(H^1_{\ccF^*_{\ccL_\ac}}\left(K,(\bbT_\chi^\ac)^\vee(1)\right)^\iota\right)^\vee $ is torsion.

\vspace{1.2mm}

Let us denote by $ \widetilde{R\Gamma}_{f}\left(G_{K,S},\bbT_\chi^{\ac,\iota};\Delta_\ccL\right)  $ the Selmer complex which is given by the Greenberg local conditions corresponding to the $\ccL_\ac$-modified local conditions in \cite[Definition 2.14 (ii)]{buy14} so that
 \begin{equation*}
       \widetilde{H}^2_{f}\left(G_{K,S},\bbT_\chi^{\ac,\iota};\Delta_\ccL\right)      \cong  \left( H_{\ccF_{\ccL_\ac}^*}^1\left(K,(\bbT_\chi^\ac)^\vee(1)\right)^\iota\right)^{\vee}.
    \end{equation*}
    
The assertion $ \widetilde{H}_{f}^3\left(G_{K,S},\bbT_\chi^{\ac,\iota};\Delta_\ccL\right) =0$ can be proven following the proof of \cite[Proposition~2.7]{bs23} \textit{verbatim}. It now follows from \cite[Proposition 8.10.13 (ii)]{Nek06} that
\begin{align*}
     \widetilde{H}_{f}^2\left(G_{K,S},\bbT_\chi^{\ac,\iota};\Delta_\ccL\right) \Big/I_\ac\cdot \widetilde{H}_{f}^2\left(G_{K,S},\bbT_\chi^{\ac,\iota};\Delta_\ccL\right) &\cong \widetilde{H}_{f}^2\left(G_{K,S},T_\chi;\Delta_\ccL\right)\\
     &\cong\left( H_{\ccF_{\ccL}^*}^1\left(K,T_\chi^\vee(1)\right)^\iota\right)^{\vee} ,
\end{align*}
where $I_\ac:={\rm ker}(\Lambda(\Gamma_\ac)\longrightarrow \frakO)$ is the augmentation ideal. Therefore, Nakayama's lemma implies that $ \widetilde{H}_{f}^2\left(G_{K,S},\bbT_\chi^{\ac,\iota};\Delta_\ccL\right) $ is $\Lambda(\Gamma_\ac)$-torsion as $\left( H_{\ccF_{\ccL}^*}^1\left(K,T_\chi^\vee(1)\right)^\iota\right)^{\vee}$ is finite by assumption, which completes the proof.
\end{proof}

\section{Anticyclotomic main conjectures}\label{anticyclotomicmainconjecture}

In this section, we will prove the main theorems which can be regarded as the algebraic formulation of the anticyclotomic main conjecture.

\begin{definition} ---
Recall the augmentation map  $\mathbbm{1}_{\rm cyc}:R_{\rm cyc}\longrightarrow R$. Suppose $R$ is a Krull unique factorization domain and $J\subset R_{\rm cyc}$ is an ideal. We set
\[\partial^*_{\rm cyc}J:=\begin{cases}
\mathbbm{1}_{\rm cyc}(J),& \; \text{if} \; (\gamma_{\rm cyc}-1)\nmid J\\
\mathbbm{1}_{\rm cyc}((\gamma_{\rm cyc}-1)^{-1}J),&\; \text{if}\; (\gamma_{\rm cyc}-1)\;\mid \; J
\end{cases}
\]
and call it the mock leading term for $J$. In particular, for a given element $\bbA:=\sum_{n\geq 0}A_n\cdot(\gamma_\cyc-1)^n\in R_{\rm cyc}$ (so that we have $A_n\in R$), we set
    \begin{equation*}
 \partial^*_{\rm cyc}\bbA:=\begin{cases}A_0, \;\;& {\rm if}\;\mathbbm 1_\cyc(\bbA)\neq 0 ,\\ 
    A_1, \;\;& {\rm if}\; \mathbbm 1_\cyc(\bbA)= 0.     \end{cases}
\end{equation*}

Further, for a torsion $R_{\rm cyc}$-module $M$ we have
\[\partial^*_{\rm cyc}{\rm char}_{R_{\rm cyc}}(M)=\begin{cases}
\partial_{\rm cyc}{\rm char}_{R_{\rm cyc}}(M),& \; \text{if} \; r(M)\leq 1\\
0,&\; \text{if}\; r(M)>1.
\end{cases}
\]
\end{definition}

\vspace{1.2mm}

\begin{theorem}\label{main_theoremB} --- Assume that Hypothesis \ref{hypothesisonfiniteness} holds. If ${\rm Reg}(\bbT_\chi^{\rm ac,\iota})\neq 0$ then we have 
\begin{equation*}
  {\rm char}_{\Lambda(\Gamma_{\rm ac})}\Big(\widetilde{H}_{f}^2\left(G_{K,S},\bbT_\chi^{\ac,\iota};\Delta_\Sigma\right)_{\rm tor} \Big)  \cdot{\rm Reg}\left(\bbT_\chi^{\rm ac,\iota}\right)= \partial^*_{\rm cyc}{\rm char}_{\Lambda(\Gamma_K)}\Big(\widetilde{H}_{f}^2\left(G_{K,S},\bbT_\chi^\iota;\Delta_\Sigma\right) \Big).
\end{equation*}
\end{theorem}
\begin{proof}
      Once we verify that the following properties are satisfied, it follows from Proposition~\ref{abstract_BSD_formulaeprop} that the asserted equality holds.
    \begin{enumerate}[label={(\arabic*)}]
    \item We have that
    \begin{align*}
   \widetilde{H}_{f}^0\left(G_{K,S},\bbT_\chi^{\ac,\iota};\Delta_\Sigma\right)&=\widetilde{H}_{f}^3\left(G_{K,S},\bbT_\chi^{\ac,\iota};\Delta_\Sigma\right)=0\\
    \widetilde{H}_{f}^0\left(G_{K,S},\bbT_\chi^\iota;\Delta_\Sigma\right)&=\widetilde{H}_{f}^3\left(G_{K,S},\bbT_\chi^\iota;\Delta_\Sigma\right)=0
\end{align*}
    \item ${\rm Tam}_v(\bbT_\chi^{\rm ac,\iota},P)=0$ for every height-one prime ideal $P$ of $\Lambda(\Gamma_{\rm ac})$ for all $v\nmid p$ in $S_f$.
    \item The $\Lambda(\Gamma_K)$-modules $ \widetilde{H}_{f}^1\left(G_{K,S},\bbT_\chi^\iota;\Delta_\Sigma\right)$ and $ \widetilde{H}_{f}^2\left(G_{K,S},\bbT_\chi^\iota;\Delta_\Sigma\right)$ are torsion.
    \item If $r\left( \widetilde{H}_{f}^2\left(G_{K,S},\bbT_\chi^\iota;\Delta_\Sigma\right)\right)>1$ then ${\rm Reg}(\bbT_\chi^{\rm ac,\iota})=0$.
\end{enumerate}

Property $(1)$ follows from Lemma~\ref{rank_zero_lemma} and Property $(2)$ is due to Lemma~\ref{tamagawa_ac}. Property $(1)$ and Lemma~\ref{euler_characteristic} imply that $\widetilde{h}^1=\widetilde{h}^2$. Therefore, Lemma \ref{reformulationofIMClemma} shows that the $\Lambda(\Gamma_K)$-modules $ \widetilde{H}_{f}^1\left(G_{K,S},\bbT_\chi^\iota;\Delta_\Sigma\right)$ and $ \widetilde{H}_{f}^2\left(G_{K,S},\bbT_\chi^\iota;\Delta_\Sigma\right)$  are torsion.

\vspace{1.2mm}

We will complete the verification of Property $(4)$ by employing \cite[Proposition 11.7.6(ii)]{Nek06}. Recall that $\gamma_{\rm cyc}\in \Gamma_{\rm cyc}$ is a fixed topological generator, and $\gamma:=\gamma_{\rm cyc}^{(\Gamma_{\rm ac})}\in \Gamma_K$ its lift. Since $\Lambda(\Gamma_{\rm ac})$ is an integral domain, the set denoted by $Q$ in \cite[\S 11]{Nek06} contains only the zero ideal $(0)$ of $\Lambda(\Gamma_{\rm ac})$. In this case, the ideal $(\gamma-1)\subset\Lambda(\Gamma_K)$ corresponds to $\overline{\fraq}$ in the notation of Nekov\'a\v r.

\vspace{1.2mm}

It now follows from \cite[Proposition 11.7.6 (vii)]{Nek06} that
\begin{equation*}
    {\rm length}_{\Lambda(\Gamma_K)_{(\gamma-1)}}\big(\widetilde{H}_{f}^2\left(G_{K,S},\bbT_\chi^\iota;\Delta_\Sigma\right)_{(\gamma-1)} \big) \geq {\rm length}_{\Lambda(\Gamma_{\rm ac})_{(0)}}\big(\widetilde{H}_{f}^1\left(G_{K,S},\bbT_\chi^{\ac,\iota};\Delta_\Sigma\right)_{(0)} \big)
\end{equation*}
with the equality if and only if ${\rm Reg}(\bbT_\chi^{\rm ac,\iota})_{(0)}\neq 0$. The latter is equivalent to asking ${\rm Reg}(\bbT_\chi^{\rm ac,\iota})\neq 0$, since $\Lambda(\Gamma_{\rm ac})$ is an integral domain. Therefore; 
\begin{align*}
    r\left(\widetilde{H}_{f}^2\left(G_{K,S},\bbT_\chi^\iota;\Delta_\Sigma\right)\right)&={\rm length}_{\Lambda(\Gamma_K)_{(\gamma-1)}}\big(\widetilde{H}_{f}^2\left(G_{K,S},\bbT_\chi^\iota;\Delta_\Sigma\right)_{(\gamma-1)} \big)\\
    &\geq {\rm length}_{\Lambda(\Gamma_{\rm ac})_{(0)}}\big(\widetilde{H}_{f}^1\left(G_{K,S},\bbT_\chi^{\ac,\iota};\Delta_\Sigma\right)_{(0)} \big)\\
    &={\rm rank}_{\Lambda(\Gamma_{\rm ac})}\left(\widetilde{H}_{f}^1\left(G_{K,S},\bbT_\chi^{\ac,\iota};\Delta_\Sigma\right)\right)
\end{align*}
with equality if and only if ${\rm Reg}(\bbT_\chi^{\rm ac})\neq 0$. Since ${\rm rank}_{\Lambda(\Gamma_{\rm ac})}\left(\widetilde{H}_{f}^1\left(G_{K,S},\bbT_\chi^{\ac,\iota};\Delta_\Sigma\right)\right)\leq 1$ (cf. Proposition~\ref{prop_rank}) and $r\left(\widetilde{H}_{f}^2\left(G_{K,S},\bbT_\chi^{\ac,\iota};\Delta_\Sigma\right)\right)>1$, the inequality is strict and hence, ${\rm Reg}(\bbT_\chi^{\rm ac,\iota}) = 0$.
\end{proof}



We write the power series expansion
\begin{equation*}
    \ccL_{p,\Sigma}^\chi:= \ccL_0+\ccL_1\cdot(\gamma_\cyc-1)+\ccL_2\cdot(\gamma_\cyc-1)^2+\dots \in \Lambda_{\ccW}(\Gamma_{\rm ac}) \widehat{\otimes} \Lambda_{\ccW}(\Gamma_{\rm cyc})
\end{equation*}
where $\ccL_i$ are elements of $\Lambda_\ccW(\Gamma_\ac)$.

\begin{theorem}\label{anticycmainconjecture} ---
    Assume the hypotheses of Theorem \ref{hsiehstheorem} as well as Hypothesis \ref{hypothesisonfiniteness}. If $\varphi\overline{\varphi}=\bbN^{k-1}$,  ${\rm Reg}(\bbT_\chi^{\rm ac,\iota})\neq 0$ and the sign in the functional equation of $f$ is $-1$, then we have
    \begin{equation*}
    {\rm char}_{\Lambda_{\ccW}(\Gamma_{\rm ac})}\Big(\widetilde{H}_{f}^2\left(G_{K,S},\bbT_\chi^{\ac,\iota};\Delta_\Sigma\right)_{\rm tor} \Big)\cdot{\rm Reg}(\bbT_\chi^{\rm ac,\iota})=\left(\ccL_1\right).
\end{equation*}
\end{theorem} 
\begin{proof}
It follows from Theorem~\ref{main_theoremB} that 
\begin{equation*}
    {\rm char}_{\Lambda(\Gamma_{\rm ac})}\Big(\widetilde{H}_{f}^2\left(G_{K,S},\bbT_\chi^{\ac,\iota};\Delta_\Sigma\right)_{\rm tor} \Big)  \cdot{\rm Reg}\left(\bbT_\chi^{\rm ac,\iota}\right)=\left(\partial^*_{\rm cyc}\ccL_{p,\Sigma}^\chi\right).
\end{equation*}
In view of
the functional equation of the Hecke $L$-function, the root number condition forces all the Hecke $L$-values appearing in the interpolation property of $\mathbbm 1_\cyc(\ccL_{p,\Sigma}^\chi)$ to vanish. Accordingly, $\mathbbm 1_\cyc(\ccL_{p,\Sigma}^\chi)$ vanishes. This also follows from the functional equation of the anticyclotomic Katz $p$-adic $L$-function (cf. \cite[\S 5]{ht93}).
\end{proof}
\begin{remark} ---
    One anticipates that the height pairing $\frah_{X}^{\rm Nek}$ is always non-degenerate (i.e. ${\rm Reg}\left(\bbT_\chi^{\rm ac,\iota}\right)\neq 0$); see \cite{Bur15,BD20} for progress in this direction.
\end{remark}

\section{Hecke characters attached to abelian varieties with complex multiplication}\label{heckecharofabelianvarieties}

Let $L$ be a finite abelian extension of the CM field $K$ with $\Delta:=\gal(L/K)$.  Fix an odd rational prime $p$ that is not ramified in $L/\bbQ$ and not dividing $|\Delta|$. 
Recall that $[K:\bbQ]=2g$.

\vspace{1.2mm}

Let $A$ be a polarized simple abelian variety of dimension $g$ defined over $L$ with CM-type $(K,\Sigma)$, i.e. there exists $\iota:K\xrightarrow{\;\sim\;} {\rm End }(A)_\bbQ$. Assume that $A$ has good ordinary reduction at $p$. 
We have the $p$-adic Tate module $T_p(A):=\varprojlim_nA[p^n]$ of A, which is a free $\bbZ_p$-module of rank $2g$ on which $G_L$ acts continuously. For the Tate module $T_p(A)$, we have a decomposition 
\begin{equation*}
    T_p(A):=\bigoplus_{\frap\mid p}T_\frap(A),
\end{equation*}
where the product is over the prime ideals of $K$ lying above $p$ and each $T_\frap(A)=\varprojlim A[\frap^n]$ is a free $\ccO_{K_\frap}$-module of rank one, see \cite[Page 502]{JT68}. The  action of $G_L$ on $T_\frap(A)$ gives rise to a character
\begin{equation*}
    \rho_\frap:G_L \longrightarrow \ccO_{K_\frap}^\times.
\end{equation*}


The theory of complex multiplication gives rise to an algebraic Hecke character $\rho: \bbAt_L\longrightarrow K^{\times}$ attached to $A$ such that, for every prime $\frap$ of $K$ lying above $p$, the $\frap$-adic avatar $\rho_{\frap}:G_L \longrightarrow \ccO_{K_\frap}^\times$ gives the action of $G_L$ on $T_\frap(A)$. We assume the following hypothesis:

     \begin{equation}
         \tag{\textbf{Shi.}} \text{ The field } L(A_{\rm tor}) \text{ is an abelian extension of } K.  \label{itm:h.shi} 
     \end{equation}
       
It follows from \cite[Theorem 7.44]{shi71} that the assumption \eqref{itm:h.shi} is equivalent to the existence of a Hecke character $\varphi$ of $K$ such that
\begin{equation}
    \rho=\varphi\circ N_{L/K}.
\end{equation}

Recall that we fixed embeddings $\iota_\infty:\overline{\bbQ}\hookrightarrow\bbC$ and $\iota_p:\overline{\bbQ}\hookrightarrow\bbC_p$ and an isomorphism $\iota:\bbC\cong \bbC_p$ such that $\iota\circ \iota_\infty=\iota_p$ and that $\Sigma$ is the set that contains precisely half of the embeddings of $K$ into $\overline{\bbQ}$. Since we are under the assumption \eqref{eqn:ordinaryassumption}, the theory of complex multiplication identifies the two sets $\{\rho_\wp^\sigma\}_{\wp,\sigma}$ and $\{\rho_\tau^{(p)}\}_\tau$, where we have
\begin{equation*}
    \rho_\wp^\sigma: \bbAt_L\big/L^\times \xrightarrow[]{\rho_\wp\;\circ\;{\rm rec}} \ccO_\wp^\times \xhookrightarrow{\;\sigma\;}\overline{\bbQ}^\times_p\subset \bbC_p^\times,
\end{equation*}
and
\begin{equation*}
    \rho_\tau^{(p)}:\bbAt_L\big/L^\times \xrightarrow[]{\;\;\rho_\tau\;\;} \bbC^\times \stackrel{\iota}{\cong}\bbC_p^\times,
\end{equation*}
where ${\rm rec}: \bbAt_L\big/L^\times\longrightarrow G_L$ denotes the reciprocity map.

\vspace{3 mm}

Fix $\tau\in\Sigma$ and identify $K$ with $K^\tau$. This choice in turn fixes a prime $\wp\in \Sigma_p$ and $\sigma:K_\wp\hookrightarrow\overline{\bbQ}_p$ in a way that $\rho_\tau^{(p)}=\rho_\wp^\sigma$. Set $\frakO:=\sigma(\ccO_\wp)$. Let $\frap=\wp^\sigma$ denote its maximal ideal, $\frakF={\rm Frak}(\frakO)$ its fraction field. By abuse of notation, we define
\begin{equation*}
    \psi:=\rho_\wp^\sigma:G_L \xtwoheadrightarrow{\;\;} \frakO^\times.
\end{equation*}
 
 \begin{definition} ---
Define the $G_L$-representation $T_{\chi,L}=\frakO(1)\otimes\chi^{-1}$, where we write $\chi=\omega_\psi=\psi\cdot\langle\psi\rangle^{-1}:G_L\longrightarrow \frakO^\times$ similar to \S~\ref{padicheckechar}. 
 \end{definition}

 We then obtain a $G_K$-representation by setting
\begin{equation*}
    \ind_L^K(T_{\chi,L}):=\frakO[G_K]\otimes_{G_L}T_{\chi,L}\cong \frakO[\Delta]\otimes_{\frakO}T_{\chi,L}.
\end{equation*}
The following lemma will provide us the relation between the $\frakO$-modules $\ind_L^K(T_{\chi,L})$ and twists $T_\chi\otimes\eta$ of $T_{\chi,L}$ by $\eta\in \widehat{\Delta}$.
 
\begin{lemma} ---
The map $f: \left(\frakO[G_K]\otimes_{\frakO[G_L]}T_{\chi,L}\right)^\eta \longrightarrow T_\chi\otimes\eta$ given by $\sigma \otimes t \mapsto \eta(\sigma)t$ is an isomorphism between $\frakO[G_K]$-modules.
\end{lemma}

\begin{proof}
Note that the map $f$ is non-trivial and both $\frakO$-modules are of rank $1$. Therefore, it suffices to check whether the morphism $f$ is $G_K$-equivariant. Let $g\in G_K$ be given. Then for $g'\in G_K$ and $t\in \frakO$, we have
\begin{align*}
    f(g\cdot(g'\otimes t))&= f(g'g\otimes t)\\
                          &= \eta(g'g)t\\
                           &= \eta(g)\eta(g')t\\
                          &= g f(g'\otimes t),
\end{align*}
which completes the proof.
\end{proof}

Fix a character $\eta\in \widehat{\Delta}$. Write $\chi_\eta=\chi\cdot \eta^{-1}$ so that we have  the $G_K$-module $T_{\chi_\eta}=\frakO(1)\otimes\chi_\eta^{-1}=T_{\chi,L}\otimes\eta$. Hence, since $p\nmid|\Delta|$ we deduce that by Shapiro's lemma,
\begin{equation*}
    \ind_L^K(T_{\chi,L})\cong {\displaystyle\bigoplus_{\eta\in \widehat{\Delta}}}T_{\chi_\eta} \text{ and } H^1(L,T_{\chi,L})\cong{\displaystyle\bigoplus_{\eta\in \widehat{\Delta}}}H^1(K,T_{\chi_\eta}).
\end{equation*}

Let $\fraf$ be the product of the ideals ${\rm Cond}(L/K)$ and $\fraf_{\psi}$ of $K$. Let $L_\infty:=LK_\infty$ denote the maximal $\bbZ_p$-power extension $L$ in $K(\fraf p^\infty)$. Note that $\Lambda(\Gamma_L) = \frakO[[\Gamma_L ]]$ with $\Gamma_L:=\gal(L_\infty/L)$ is also isomorphic to the formal power series ring $\Lambda:=\frakO[[X_1,\dots, X_{g+1}]]$.

\begin{lemma}\label{selmergp_isom} ---
   We have the following isomorphism between the modified Selmer groups
    \begin{equation*}
        H^1_{\ccF_\Sigma}\left(L, T_{\chi,L}\otimes\Lambda(\Gamma_L)\right)\cong \displaystyle{\bigoplus_{\eta\in\widehat{\Delta}}} H^1_{\ccF_\Sigma}\left(K,T_{\chi_\eta}\otimes\Lambda(\Gamma_K)\right).
    \end{equation*}
\end{lemma}
\begin{proof}
   The result follows immediately from the comparison of the local conditions.
\end{proof}

\begin{theorem}\label{maintheorem} ---
Assume that the hypotheses of Theorem \ref{hsiehstheorem} along with \eqref{itm:h.shi}. Then, for each $\eta\in\widehat{\Delta}$, the module $H^1_{\ccF_\Sigma^*}(K,\bbT_{\chi_\eta}^\vee(1))$ is $\Lambda$-cotorsion and we have

\begin{align*}
        {\rm char}_{\Lambda_\ccW}\left( H^1_{\ccF_\Sigma^*}\left(L,\bbT_{\chi,L}^\vee(1)\right)^\vee  \right)&={\displaystyle\prod_{\eta\in \widehat{\Delta}}} {\rm char}_{\Lambda_\ccW}\left( H^1_{\ccF_\Sigma^*}(K,\bbT_{\chi_\eta}^\vee(1))^\vee \right)\\
    &={\displaystyle\prod_{\eta\in \widehat{\Delta}}} \left(\ccL_{p,\Sigma}^{\chi_\eta}\right) .
\end{align*}
\end{theorem}

\begin{proof}
   The assertion follows immediately from Theorem~\ref{hsiehstheorem} and Lemma~\ref{selmergp_isom}.
\end{proof}

\vspace{1.2mm}

\begin{definition} ---
    For $X=L$ or $L_\infty$, define the $\frap$-relaxed Selmer group $\Sel'_{\frap}(A/X)$ attached to $A$ by setting
    \begin{equation*}
        \Sel'_\frap(A/X):=\ker\left(H^1(X,A[\frap^\infty])\longrightarrow {\displaystyle\prod_{v\nmid \frap}}\dfrac{H^1(X_v,A[\frap^\infty])}{\im\left(A(X)\otimes \frakF/\frakO \xhookrightarrow{\kappa_v} H^1(X_v,A[\frap^\infty])\right)} \right),
    \end{equation*}
where $\kappa_v$ denotes the Kummer map.
\end{definition}

For $\frap$ as above, we assume the following non-anomaly condition on $A$:
\begin{equation}\label{non-anomaly_cond}
    A(L_v)[\frap]=0 \text{ for every prime } v \text{ of } L \text{ above } p.
\end{equation}


\begin{prop}\label{control_theorem} --- We have the following inclusions:
\begin{enumerate}[label=(\roman*)]
    \item $ \Sel'_\frap(A/L) \xhookrightarrow{\;\;\;}  \Sel'_\frap(A/L_\infty)^{\Gamma_L}$.
    \item $ \Sel'_\frap(A/L_\infty) \subset  H^1_{\ccF_\Sigma^*}\left(L,\left(T_{\chi,L}\otimes\Lambda(\Gamma_L)\right)^\vee(1)\right) $.
\end{enumerate}
\end{prop}

\begin{proof} The proof is similar to \cite[Proposition 4.14]{buy14}.
\end{proof}

Since $L$ contains the reflex field of $K$, it follows that the Hasse-Weil $L$-series $L(A/L,s)$ of $A$ factors into a product of Hecke $L$-series
\begin{equation*}
 L(A/L,s)={\displaystyle\prod_{\tau\in I_K}}L(\psi_\tau,s), 
 \end{equation*}
where $I_K$ denotes all the embeddings of $K$ into $\overline{\bbQ}$. By making use of Theorem~\ref{maintheorem} as well as the interpolation property of the $p$-adic $L$-function, we deduce the following result regarding the Mordell-Weil group $A(L)$ of the abelian variety $A$.

\vspace{1.2mm}

Recall that we fixed an embedding $\tau\in\Sigma$ at the beginning of \S~\ref{heckecharofabelianvarieties}, which in turn fixes a prime ideal $\frap\in \Sigma$.

\begin{theorem}\label{lseriesthm} ---
      Assume that the hypotheses of Theorem~\ref{maintheorem} and the non-anomaly condition hold true. If $L(A/L, 1)\neq 0$, then both $A(L)$ and $\Sha(A/L)[\frap^\infty]$ are finite.
\end{theorem}
\begin{proof}
 The assertion follows from Theorem \ref{maintheorem}, Proposition~\ref{control_theorem}, and the interpolation property of Katz $p$-adic $L$-functions given in \cite{ht93}. 
\end{proof}

\bibliographystyle{alpha} 
\bibliography{reference}

\end{document}